\documentclass[11pt,a4paper]{article}
\usepackage{amsfonts,amsgen,amstext,amsbsy,amsopn,amsfonts,amssymb,amscd}
\usepackage[leqno]{amsmath}
\usepackage[amsmath,amsthm,thmmarks]{ntheorem}
\usepackage{epsf,epsfig}
\usepackage{float}
\usepackage{dsfont}
\usepackage{ebezier,eepic}
\usepackage{color}
\usepackage{tikz}
\usepackage{multirow}
\usepackage{mathrsfs}
\usepackage{graphicx}
\usepackage{subfigure}
\newtheorem{thm}{Theorem}[section]

\newtheorem{prop}[thm]{Proposition}

\newtheorem{lem}[thm]{Lemma}

\newtheorem{false statement}{False statement}
\newtheorem{cor}[thm]{Corollary}
\newtheorem{fact}[thm]{Fact}

\theoremstyle{definition}
\newtheorem{defn}[thm]{Definition}
\newtheorem{claim}[thm]{Claim}

\newtheorem{conj}[thm]{Conjecture}

\makeatletter \@addtoreset{equation}{section}

\def\hh{\mathcal{H}}

\def\hl{\mathcal{L}}
\def\hht{\mathcal{T}}
\def\hd{\mathcal{D}}

\def\hf{\mathcal{F}}
\def\hg{\mathcal{G}}
\def\hk{\mathcal{K}}
\def\ha{\mathcal{A}}
\def\hb{\mathcal{B}}
\def\hs{\mathcal{S}}

\begin{document}
\title{\bf\Large The overflow in the Katona Theorem}
\date{}
\author{Peter Frankl$^1$, Jian Wang$^2$\\[10pt]
$^{1}$R\'{e}nyi Institute, Budapest, Hungary\\[6pt]
$^{2}$Department of Mathematics\\
Taiyuan Univer sity of Technology\\
Taiyuan 030024, P. R. China\\[6pt]
E-mail:  $^1$frankl.peter@renyi.hu, $^2$wangjian01@tyut.edu.cn
}

\maketitle
\begin{abstract}
Let $n>2r>0$ be integers. We consider families $\hf$ of subsets of an $n$-element set, in which the union of any two members has size  at most $2r$. One of our results states that for $n\geq 6r$ the number of members of size exceeding $r$ in $\hf$ is at most $\binom{n-2}{r-1}$. Another result shows that for $n>3.5r$ the number of sets of size at least $r$ is at most $\binom{n}{r}$. Both bounds are best possible and the latter sharpens the classical Katona Theorem. Similar results are proved for the odd case of the Katona Theorem as well.

\vspace{6pt}

\vspace{6pt}
\end{abstract}
\section{Introduction}
Let $[n]=\{1,2,\ldots,n\}$ be the standard $n$-element set and $2^{[n]}$ its power set. For $0\leq k\leq n$ let $\binom{[n]}{k}\subset 2^{[n]}$ be the collection of all $k$-subsets. Subsets of $2^{[n]}$ are called {\it families}. If $\hf\subset \binom{[n]}{k}$ it is called $k$-uniform.

For a positive integer $t$, the family $\hf$ is called {\it $t$-intersecting} if $|F\cap F'|\geq t$ for all $F,F'\in \hf$.  Also for a positive integer $u$, $\hf$ is called {\it $u$-union} if $|F\cup F'|\leq u$ for all $F,F'\in \hf$.

Define $\hf^c=\{[n]\setminus F\colon F\in \hf\}$, the {\it family of complements}. It should be obvious that $\hf\subset 2^{[t]}$ is $t$-intersecting if and only if  $\hf^c$ is $(n-t)$-union.

Let us state the two most important results concerning intersecting families, the Erd\H{o}s-Ko-Rado Theorem and the Katona Theorem.

\begin{thm}[\cite{EKR}]
Let $n>k\geq t>0$ be integers and  $\hf\subset \binom{[n]}{k}$ be $t$-intersecting. Then for $n>n_0(k,t)$,
\begin{align}\label{ineq-ekr}
|\hf|\leq \binom{n-t}{k-t}.
\end{align}
\end{thm}

Let us mention that the exact value $n_0(k,t)=(k-t+1)(t+1)$ was determined by Erd\H{o}s, Ko and Rado for $t=1$, by Frankl \cite{F78} for $t\geq 15$ and finally Wilson \cite{W84} closed the gap $2\leq t\leq 14$ with a completely different proof valid for all $t$. Let us note also that the {\it full $t$-star}
\[
\hs(n,k,t) = \left\{S\in \binom{[n]}{k}\colon [t]\subset S\right\}
\]
shows that \eqref{ineq-ekr} is best possible.

The corresponding problem for {\it non-uniform} families was solved by Katona. Let us state it for $u$-union families.

\begin{thm}[\cite{Katona}]
Let $n>u>0$ and suppose that $\hf\subset 2^{[n]}$ is $u$-union. Then (i) or (ii) holds.
\begin{itemize}
  \item[(i)] $u=2d$ is even and
  \begin{align}\label{ineq-katona1}
  |\hf| \leq \sum_{0\leq i\leq d}\binom{n}{i}.
  \end{align}
  \item[(ii)] $u=2d+1$ is odd and
  \begin{align}\label{ineq-katona2}
  |\hf| \leq 2\sum_{0\leq i\leq d}\binom{n-1}{i}.
  \end{align}
\end{itemize}
\end{thm}

Let us define the Katona Families $\hk(n,u)$:
\begin{align*}
\hk(n,2d)=\{K\subset [n]\colon |K|\leq d\},\
\hk(n,2d+1)=\{K\subset [n]\colon |K\setminus \{1\}|\leq d\}.
\end{align*}

These families show that both \eqref{ineq-katona1} and \eqref{ineq-katona2} are best possible. As Katona \cite{Katona} showed, for $u\leq n-2$, $u=2d$ the family $\hk(n,2d)$ is the unique family attaining equality in \eqref{ineq-katona1}. Similarly, for $u\leq n-2$, $u=2d+1$, $\hk(n,2d+1)$ is unique up to isomorphism.

Let us mention that in the case of $u=n-1$ both bounds \eqref{ineq-katona1} and \eqref{ineq-katona2} reduce to $|\hf|\leq 2^{n-1}$ which was already known to Erd\H{o}s, Ko and Rado. In this case there are very many non-isomorphic extremal families (cf. \cite{EKR}).

An intersecting family $\hf\subset \binom{[n]}{k}$ is called {\it non-trivial} if $\cap \hf=\emptyset$. Define the Hilton-Milner family  $\hh(n,k)$ and the Triangle family $\hht(n,k)$ as:
\[
\hh(n,k)= \left\{H\in \binom{[n]}{k}\colon 1\in H,\ H\cap [2,k+1]\neq \emptyset\right\}\cup \{[2,k+1]\}
\]
and
\[
\hht(n,k)= \left\{F\in \binom{[n]}{k}\colon |F\cap [3]|\geq 2\right\}.
\]
It is easy to see that they are both non-trivial intersecting.

For the $t=1$, $n>2k$ case of the  Erd\H{o}s-Ko-Rado Theorem, Hilton and Milner proved a very important stability result.

\begin{thm}
Let $n> 2k$ and let $\hf\subset \binom{[n]}{k}$ be non-trivial intersecting, then
\begin{align}\label{ineq-nontrival}
|\hf| \leq \binom{n-1}{k-1}-\binom{n-k-1}{k-1}+1.
\end{align}
\end{thm}

The construction $\hh(n,k)$ shows that \eqref{ineq-nontrival} is best possible.

The corresponding statement for the Katona Theorem was proved by Frankl \cite{F17}.  Define $\hk^*(n,u)$ as:
\[
\hk^*(n,2d) =\hk(n,2d)\setminus \left\{H\in \binom{[n]}{d}\colon H\cap [d+1]=\emptyset\right\}\cup \{[d+1]\}
\]
and
\[
\hk^*(n,2d+1) =\hk(n,2d+1)\cup \left\{H\in \binom{[n]}{d+1}\colon 1\in H, H\cap [2,d+2]=\emptyset\right\}\cup \{[2,d+2]\}.
\]

\begin{thm}[\cite{F17}]
Suppose that $\hf\subset 2^{[n]}$ is $u$-union, $2\leq u< n$ and $\hf\not\subset \hk(n,u)$ holds. Then
\[
|\hf|\leq |\hk^*(n,u)|.
\]
\end{thm}

Let us note that $|\hh(n,k)\setminus \hs(n,k,1)|=1$ and $|\hk^*(n,u)\setminus \hk(n,u)|=1$, i.e., both families are in a sense extremely close to the original optimal families. For $x\in [n]$, define \[
\hf(\bar{x})=\{F\in \hf\colon x\notin \hf\}.
\]
Let us define the {\it diversity} $\gamma(n,k)$ for $n>2k$ by
\[
\gamma(n,k) =\max_{\hf}\min_{x\in [n]} |\hf(\bar{x})|,
\]
where the maximum is taken over all 1-intersecting families $\hf\subset \binom{[n]}{k}$. The problem of determining $\gamma(n,k)$ was raised by Katona (cf. \cite{LP}). Improving earlier results (\cite{LP}, \cite{F17-2}, \cite{Ku1}), it is proved that

\begin{thm}[\cite{F2020}, \cite{FW2022-3}]\label{thm-diveristy}
For $n>36k$,
\begin{align}\label{ineq-diversity}
\gamma(n,k) = \binom{n-3}{k-2}.
\end{align}
\end{thm}

In case of the Katona Theorem, for $u=2d$, there is no special element in $\hk(n,u)$. Therefore we prefer a new name for the maximum of $|\hf\setminus \hk(n,u)|$.

\begin{defn}
For $n>2d+1$ define the {\it overflow} $\sigma_{2d}(\hf)$ of $\hf$ as $|\hf\setminus \hk(n,2d)|$ and set
\[
\sigma(n,2d) = \max\left\{\sigma_{2d}(\hf)\colon \hf\subset 2^{[n]} \mbox{ is $2d$-union}\right\}.
\]
\end{defn}

A natural family  with relatively large overflow is
\[
\hb(n,2d)=\{B\subset [n]\colon |B\setminus [2]|\leq d-1\}.
\]
Obviously,
\[
\hb(n,2d)\setminus \hk(n,2d)=\left\{B\in \binom{[n]}{d+1}\colon [
2]\subset B\right\}.
\]
Our main result shows that $\sigma(n,2d)=\binom{n-2}{d-1}$ for $n\geq 6d$.

\begin{thm}\label{thm-main}
Suppose that $n\geq 6d$, $\hf\subset 2^{[n]}$ is $2d$-union. Then
\begin{align}\label{ineq-main}
|\hf\setminus \hk(n,2d)|\leq \binom{n-2}{d-1}.
\end{align}
\end{thm}

Recall the standard notations
\[
\hf^{(\ell)}=\{F\in \hf\colon |F|=\ell\} \mbox{ and } \hf^{(\geq \ell)}=\{F\in \hf\colon |F|\geq \ell\}.
\]
With these notations $\hf\setminus \hk(n,2d)=\hf^{(\geq d+1)}=\hf^{(d+1)}\cup\hf^{(d+2)}\cup  \ldots\cup \hf^{(n)}$. Moreover, if $\hf$ is $u$-union then $\hf^{(\ell)}=\emptyset$ for $\ell>u$.

Let us define one more $2d$-union family:
\[
\hd=\{D\subset [n]\colon |
D\setminus [4]|\leq d-2\}.
\]
Its overflow is
\[
|\hd\setminus \hk(n,2d)|=|\hd^{(d+1)}|+|\hd^{(d+2)}|=5\binom{n-4}{d-2}+\binom{n-4}{d-3}.
\]
For $n\leq 4d-2$,
\begin{align*}
5\binom{n-4}{d-2}+\binom{n-4}{d-3}-\binom{n-2}{d-1}&=4\binom{n-4}{d-2}+\binom{n-3}{d-2}-\binom{n-2}{d-1}\\[5pt]
&=4\binom{n-4}{d-2}-\binom{n-3}{d-1}\\[5pt]
&=\frac{4d-1-n}{d-1}\binom{n-4}{d-2}> 0.
\end{align*}
This shows that \eqref{ineq-main} does not hold for $n<4d-1$.

Let
\[
\hk_x(n,2d+1)=\{K\subset [n]\colon |K\setminus \{x\}|\leq d\}.
\]
Define the {\it overflow} $\sigma_{2d+1}(\hf)$ of $\hf$ as $\min_{x\in [n]}|\hf\setminus \hk_x(n,2d)|$.

For the odd case, we obtain the following result.

\begin{thm}\label{thm-main2}
Suppose that $n> 36(d+1)$, $\hf\subset 2^{[n]}$ is $(2d+1)$-union. Then
\begin{align}\label{ineq-main2}
\sigma_{2d+1}(\hf)\leq 2\binom{n-3}{d-1}.
\end{align}
\end{thm}

Define
\[
\hg=\hg(n,2d+1) =\left\{G\subset [n]\colon |G\cap [4,n]|\leq d-1\right\}.
\]
Then $\hg^{(d+2)}$ is the star of $[3]$ and $\hg^{(d+1)}$ is the triangle family $\hht(n,d+1)$. Hence for $x\in [3]$,
\[
|\hg\setminus \hk_x(n,2d+1)|=2\binom{n-3}{d-1}.
\]
Thus \eqref{ineq-main2} is best possible.

Our next result is a sharpening of the Katona Theorem in the case $u=2r$.

\begin{thm}\label{prop-1.9}
Let $n\geq 3.5r+1$ and let $\hf\subset 2^{[n]}$ be $2r$-union. Then
\begin{align}\label{ineq-prop1.9}
|\hf^{(\geq r)}|\leq \binom{n}{r}
\end{align}
with equality if and only if $\hf^{(\geq r)}=\binom{[n]}{r}$.
\end{thm}

To see that Theorem \ref{prop-1.9} implies the Katona Theorem just note that $\hf^{(<r)}\leq \sum_{0\leq i<r}\binom{n}{i}$ is evident for any $\hf\subset 2^{[n]}$.

The coefficient 3.5 in the above theorem could most likely be replaced by a smaller number. However the example and computation below show that the statement is no longer true for $n<2.2r$.

Recall
the $2r$-union family:
\[
\hd=\{D\subset [n]\colon |
D\setminus [6]|\leq r-3\}.
\]
It is easy to see that $\hd^{(\geq r+4)}=\emptyset$ and
\begin{align*}
&|\hd^{(r)}| =20\binom{n-6}{r-3}+15\binom{n-6}{r-4}+6\binom{n-6}{r-5}+\binom{n-6}{r-6},\\[5pt]
&|\hd^{(r+1)}|=15\binom{n-6}{r-3}+6\binom{n-6}{r-4}+\binom{n-6}{r-5},\\[5pt] &|\hd^{(r+2)}|=6\binom{n-6}{r-3}+\binom{n-6}{r-4},\\[5pt]
&|\hd^{(r+3)}|=\binom{n-6}{r-3}.
\end{align*}
Thus,
\[
|\hd^{(\geq r)}|=42\binom{n-6}{r-3}+22\binom{n-6}{r-4}+7\binom{n-6}{r-5}+\binom{n-6}{r-6}.
\]
Since
\[
\binom{n}{r}=\binom{n-6}{r}+6\binom{n-6}{r-1}+ 15\binom{n-6}{r-2}+20\binom{n-6}{r-3}+15\binom{n-6}{r-4}
+6\binom{n-6}{r-5}+\binom{n-6}{r-6},
\]
it follows that
\begin{align}\label{ineq-1.1}
|\hd^{(\geq r)}|-\binom{n}{r}&=\binom{n-6}{r}+6\binom{n-6}{r-1}+ 15\binom{n-6}{r-2}+20\binom{n-6}{r-3}\nonumber\\[5pt]
&\qquad -22\binom{n-6}{r-3}-7\binom{n-6}{r-4}-\binom{n-6}{r-5}.
\end{align}
Note that for any  $0\leq i\leq 5$ and $n=(1+c)r$,
\[
\frac{\binom{n-6}{r-i}}{\binom{n-6}{r-i-1}}= \frac{n-r+i-5}{r-i}=\frac{cr+i-5}{r-i}\rightarrow c
\]
as $r$ tends to infinity.  Thus \eqref{ineq-1.1}$>0$ is essentially equivalent to
\[
1+7c+22c^2-15c^3-6c^4-c^5>0.
\]
 This holds for $1<c<1.2$.
Thus to establish \eqref{ineq-prop1.9} one needs at least $n\geq 2.2r$.

 For the  odd case we have the following result.

\begin{thm}\label{thm-1.10}
Let $n> 4r$ and let $\hf\subset 2^{[n]}$ be $(2r+1)$-union. Then
\begin{align}\label{ineq-prop1.10}
|\hf^{(\geq r+1)}|\leq \binom{n-1}{r}
\end{align}
with equality if and only if $\hf^{(\geq r+1)}=\hs(n,r+1,1)$ up to isomorphism.
\end{thm}

 Computation similar to the even case, using the example $\{F \in [n]: |F\setminus [5]| \leq  r-2 \}$ shows that $4r$ cannot be
 replaced by $2.88r$.

Two families $\ha\subset \binom{[n]}{a}$ and  $\hb\subset \binom{[n]}{b}$ are  called {\it cross $t$-intersecting} if $|A\cap B|\geq t$ for any $A\in \ha$ and $B\in \hb$. We need the following simple fact.

\begin{fact}\label{ineq-1.8}
Suppose that $\hf\subset 2^{[n]}$ is $(2d+h)$-union then (i), (ii) hold.
\begin{itemize}
  \item[(i)] $\hf^{(d+i)}$ is $(2i-h)$-intersecting.
  \item[(ii)] $\hf^{(d+i)}$ and $\hf^{(d+j)}$ are cross $(i+j-h)$-intersecting.
\end{itemize}
\end{fact}

\begin{proof}
Since (i) is the special case $i=j$ of $(ii)$, we only prove (ii). Choose $F\in \hf^{(d+i)}$, $G\in \hf^{(d+j)}$. Then
\[
|F\cap G|=|F|+|G|-|F\cap G| \geq (d+i)+(d+j)-(2d+h)=i+j-h.
\]
\end{proof}

Let us recall the following common notations: for $P\subset Q\subset [n]$, let
\[
\hf(P,Q) = \left\{F\setminus Q\colon F\in\hf,\ F\cap Q=P \right\}.
\]
When $P=Q$, we simply write $\hf(P)$.
\section{Shifting and random walks}

For a family $\hf\subset 2^{[n]}$ and two integers $1\leq i<j\leq k$, one define the $i\leftarrow j$ shift $S_{ij}$ by
\[
S_{ij}(\hf) =\{S_{ij}(F)\colon F\in \hf\},
\]
where
$$S_{ij}(F)=\left\{
                \begin{array}{ll}
                  F':=(F\setminus\{j\})\cup\{i\}, & j\in F, i\notin F \text{ and } F' \notin \hf; \\[5pt]
                  F, & \hbox{otherwise.}
                \end{array}
              \right.
$$
Obviously, $|S_{ij}(\hf)|=|\hf|$, $|S_{ij}(F)|=|F|$ and it is known (cf. e.g. \cite{F87}) that the $i\leftarrow j$ shift maintains the $t$-intersecting, $u$-union and cross $t$-intersecting properties.

Repeated application of the $i\leftarrow j$ shift leads to a family $\tilde{\hf}$ satisfying $S_{ij}(\tilde{\hf})=\tilde{\hf}$ for all $1\leq i<j\leq n$. Such families are called {\it initial}. This initiality is with respect to the following partial order.

Let $(x_1,\ldots,x_k)$ denote the $k$-set $\{x_1,\ldots,x_k\}$ where we know that $x_1<\ldots<x_k$. If $(x_1,\ldots,x_k)$, $(y_1,\ldots,y_k)\in \binom{[n]}{k}$ then we say that $(x_1,\ldots,x_k)$ precedes $(y_1,\ldots,y_k)$ and denote it by $(x_1,\ldots,x_k)\prec(y_1,\ldots,y_k)$ if and only if $x_i\leq y_i$ for all $1\leq i\leq k$. Then $\hf$ is initial if and only if for all $F,G\subset [n]$ with $|F|=|G|$, $G\in \hf$ and $F\prec G$ imply $F\in \hf$.

\begin{prop}[\cite{F78}]
Suppose that $F\in \hf\subset 2^{[n]}$ and $\hf$ is initial and $t$-intersecting. Then
\begin{align}\label{ineq-2.1}
|F\cap [t+2i]|\geq t+i \mbox{ for some } i\geq 0.
\end{align}
\end{prop}
\begin{prop}[\cite{F87}]
Suppose that $\hf,\hg\subset 2^{[n]}$ are cross $t$-intersecting, $F\in \hf$, $G\in \hg$. Then there exists some $s\geq t$ so that
\begin{align}\label{ineq-2.2}
|F\cap [s]|+|G\cap [s]| \geq s+t.
\end{align}
\end{prop}

One can deduce \eqref{ineq-2.1} from \eqref{ineq-2.2}. On the other hand \eqref{ineq-2.1} has a geometric interpretation. Let us associate with every subset $R\subset [n]$ a walk $w(R)$ of $n$ steps in the plane starting from the origin. If after the $i$th step we are in point $(x,y)$ then we move to $(x,y+1)$ or $(x+1,y)$ according to whether $i+1\in R$ or $i+1\notin R$. If $|R|=k$ then the walk $w(R)$ terminates in $(n-k,k)$.

Choose $i$ in \eqref{ineq-2.1} to be maximal implies equality and thereby $w(F)$ is in the vertex $(i,t+i)$ after $t+2i$ steps. That is, it is on the line $y=x+t$.

\begin{cor}\label{cor-2.3}
If $\hf\subset 2^{[n]}$ is $t$-intersecting and initial then $w(F)$ hits the line $y=x+t$ for each $F\in \hf$.
\end{cor}

The reflection of $(0,0)$ on the line $y=x+t$ is $(-t,t)$. The {\it reflection principle} says that the number of walks from $(-t,t)$ to any fixed vertex $V$ under the line is the same as the number of walks from $(0,0)$ to $V$ hitting the line. This readily implies
\begin{prop}\label{prop-2.4}
Suppose that $\hf\subset 2^{[n]}$ is $t$-intersecting. Then
\begin{align}\label{ineq-universalbound}
|\hf^{(t+\ell)}|\leq \binom{n}{\ell} \mbox{ for all }\ell\geq 0.
\end{align}
\end{prop}

Similarly, we have the following more general statement.

\begin{prop}\label{prop-2.5}
Let $(a,b)$ and $(n-k,k)$ be two points below the line $y=x+t$ and assume $a+t\leq k$. Then the number of walks from $(a,b)$ to $(n-k,k)$ hitting the line $y=x+t$ is $\binom{n-a-b}{k-a-t}$.
\end{prop}

It should be mentioned that the universal bound \eqref{ineq-universalbound} for $t$-intersecting families was improved by Frankl \cite{F2020-2}.

\begin{thm}[\cite{F2020-2}]
Suppose that $\hf\subset 2^{[n]}$ is $t$-intersecting. Then
\begin{align}\label{ineq-universalbound2}
|\hf^{(t+\ell)}|\leq \binom{n-1}{\ell} \mbox{ for }\ell\leq \frac{n-t-1}{2}.
\end{align}
\end{thm}

The following proposition was proved in \cite{F78} and we include its proof for completeness.
\begin{prop}[\cite{F78}]
Let $\hf\subset \binom{[n]}{k}$ be an initial family.  If
$(1,2,\ldots,p,p+2,p+4,\ldots,2k-p)\notin \hf$ for some $0\leq p\leq k-1$, then
\begin{align}\label{ineq-walks1}
|\hf| \leq \binom{n}{k-p-1}.
\end{align}
If
$(p,p+2,p+4,\ldots,p+2k-2)\notin \hf$ for some $2\leq p\leq k-1$, then
\begin{align}\label{ineq-walks2}
|\hf| \leq \binom{n}{k}-\binom{n-p+2}{k}+\binom{n-p+2}{k-1}.
\end{align}
\end{prop}
\begin{proof}
  If
$(1,2,\ldots,p,p+2,p+4,\ldots,2k-p)\notin \hf$, then for each $F\in \hf$, $w(F)$ hits the line $y=x+p+1$. By Proposition \ref{prop-2.5},
\begin{align*}
|\hf| &\leq \binom{n}{k-p-1}.
\end{align*}

Note that
\[
|\{F\in\hf\colon  F\cap [p-2]\neq \emptyset\}| \leq \binom{n}{k}-\binom{n-p+2}{k}.
\]
Suppose that $(p,p+2,p+4,\ldots,p+2k-2)\notin \hf$. Then $w(F)$ hits the line $y=x-p+3$ for each $F\in \hf$ with $F\cap [p-2]=\emptyset$. By Proposition \ref{prop-2.5},
\begin{align*}
&\qquad |\{F\in\hf\colon  F\cap [p-2]\neq \emptyset\}| \\[5pt]
&\leq \mbox{the number of walks from }(p-2,0) \mbox{ to } (n-k,k) \mbox{ hitting the line }y=x-p+3 \\[5pt]
&=  \binom{n-p+2}{k-1}.
\end{align*}
Thus \eqref{ineq-walks2} follows.
\end{proof}

We need the following inequality in the proofs.

\begin{prop}
Let $n, r, a, b$ be positive integers, $n\geq r+a$, $n>r>b$. Then
\begin{align}\label{ineq-key}
\frac{\binom{n-a}{r}}{\binom{n}{r-b}} \geq \left(\frac{n-r+b-a+1}{n-a+1}\right)^a\left(\frac{n-r+b-a}{r}\right)^b.
\end{align}
\end{prop}
\begin{proof}
The inequality follows directly  from
\begin{align*}
\frac{\binom{n-a}{r}}{\binom{n}{r-b}} &=\frac{(n-r+b)(n-r+b-1)\ldots(n-r+b-a+1)}{n(n-1)\ldots(n-a+1)}\\[5pt]
& \qquad \times \frac{(n-r+b-a)(n-r+b-a-1)\ldots(n-r-a+1)}{r(r-1)\ldots(r-b+1)}\\[5pt]
&\geq \left(\frac{n-r+b-a+1}{n-a+1}\right)^{a}\left(\frac{n-r+b-a}{r}\right)^b.
\end{align*}
\end{proof}

\section{Proof of Theorem \ref{thm-main}}
For $E=(a_1,a_2,\ldots,a_k)$, let $E-p$ denote the set $(a_1-p,a_2-p,\ldots,a_k-p)$. Let $\hf\subset \binom{[p+1,n]}{k}$ be an initial family. Define the {\it left-translate} $L_p(\hf)$  of $\hf$ as
\[
L_p(\hf) :=\{F-p\colon F\in \hf\}.
\]
Clearly, $L_p(\hf)\subset\binom{[n-p]}{k}$ is also an initial family.

\begin{proof}[Proof of Theorem \ref{thm-main}]
 Noting that shifting does not change the overflow $|\hf\setminus \hk(n,2d)|$, we may assume that $\hf$ is initial. Since adding a subset $F_0\subset F\in \hf$ to $\hf$ will not affect the $u$-union property, we may assume that $\hf$ is a {\it complex}, that is, $F_0\subset F\in \hf$
 implies $F_0\in \hf$.

 If $\hf^{(\ell)}=\emptyset$ for all $\ell\geq d+2$ then $\hf\setminus \hk(n,2d)=\hf^{(d+1)}$. By Fact \ref{ineq-1.8}, $\hf^{(d+1)}$ is 2-intersecting. Thus  by \eqref{ineq-ekr}, $|\hf^{(d+1)}|\leq \binom{n-2}{d-1}$ for $n>3d$ and the theorem holds.

 From now on we assume $\hf^{(\ell)}\neq \emptyset$ for some $\ell\geq d+2$. Since $\hf$ is a complex, $\hf^{(d+2)}\neq\emptyset$. By initiality $[d+2]\in \hf$ follows. By Fact \ref{ineq-1.8}, $\hf^{(d+1)}, \hf^{(d+2)}$ are cross 3-intersecting. Note that $|[d+2]\cap(1,2,d+3,d+5,\ldots,3d-1)|=2$. It follows that $(1,2,d+3,d+5,\ldots,3d-1)\notin \hf^{(d+1)}$.

 Let $\hf_0^{(d+1)}=\{F\in \hf^{(d+1)}\colon \{1,2\}\subset F\}$ and let $\hf_1^{(d+1)}=\hf^{(d+1)}\setminus \hf_0^{(d+1)}$. We distinguish two cases.

 \vspace{5pt}
 {\bf Case 1. } $(1,2,4,6,8,\ldots,2d)\notin \hf^{(d+1)}$.
 \vspace{5pt}

 Since $(1,2,4,6,8,\ldots,2d)\notin \hf^{(d+1)}$, $(2,4,6,\ldots,2d-2)\notin L_2(\hf^{(d+1)}([2]))$. By \eqref{ineq-walks2} we have
\begin{align}\label{ineq-case1-1}
|\hf_0^{(d+1)}|=|L_2(\hf^{(d+1)}([2]))|\leq  \binom{n-2}{(d-1)-1}=\binom{n-2}{d-2}.
\end{align}

Since $\hf^{(d+1)}$ is 2-intersecting, by initiality $\hf(\{\ell\},[2])$ is also $2$-intersecting for $\ell=1,2$. By \eqref{ineq-universalbound2} we have $|\hf(\{\ell\},[2])| \leq \binom{n-3}{d-2}$. The initiality also implies  that $\hf(\emptyset,[2])$ is  4-intersecting. Using \eqref{ineq-universalbound2} we have $|\hf(\emptyset,[2])| \leq \binom{n-3}{d-3}$. Thus,
\begin{align}\label{ineq-case1-2}
|\hf_1^{(d+1)}| &=|\hf(\{1\},[2])|+|\hf(\{2\},[2])|+|\hf(\emptyset,[2])|\nonumber\\[5pt]
&\leq 2\binom{n-3}{d-2}+\binom{n-3}{d-3}=\binom{n-3}{d-2}+\binom{n-2}{d-2}.
\end{align}

 By Fact \ref{ineq-1.8} and Proposition \ref{prop-2.4},
\[
|\hf^{(d+i)}| \leq \binom{n}{d+i-2i} =\binom{n}{d-i}.
\]
Note that $n\geq 6d$ implies that
\[
\frac{\binom{n}{d-i}}{\binom{n}{d-i-1}} =\frac{n-d+i+1}{d-i}> 5.
\]
It follows that
\begin{align}\label{ineq-case1-3}
\sum_{2\leq i\leq d} |\hf^{(d+i)}| \leq \sum_{2\leq i\leq d} \binom{n}{d-i} \leq \binom{n}{d-2}+\frac{5}{4}\binom{n}{d-3}&\leq \binom{n}{d-2}+\frac{5}{4}\frac{n^2}{(n-d)^2}\binom{n-2}{d-3}\nonumber\\[5pt]
&\leq \binom{n}{d-2}+\frac{9}{5}\binom{n-2}{d-3}.
\end{align}

Adding \eqref{ineq-case1-1}, \eqref{ineq-case1-2} and \eqref{ineq-case1-3}, we get
\begin{align*}
\sum_{1\leq i\leq d}|\hf^{(d+i)}| &\leq  \binom{n-3}{d-2}+2\binom{n-2}{d-2}+\binom{n}{d-2}+\frac{9}{5}\binom{n-2}{d-3}\\[5pt] &=\binom{n-3}{d-2}+\frac{1}{5}\binom{n-2}{d-2}+\frac{9}{5}\binom{n-1}{d-2}+\binom{n}{d-2}.
\end{align*}
By  $n\geq 6d$,
\begin{align*}
\sum_{1\leq i\leq d}|\hf^{(d+i)}|
&\leq \left(\frac{d-1}{n-2}+\frac{1}{5}\frac{d}{n-d}+\frac{9}{5}\frac{dn}{(n-d)^2}+\frac{d n^2}{(n-d)^3}\right)\binom{n-2}{d-1}\\[5pt]
&\leq \left(\frac{1}{6}+\frac{1}{5}\times\frac{1}{5}+\frac{9}{5}\times\frac{6}{5^2}+\frac{ 6^2}{5^3}\right)\binom{n-2}{d-1}\\[5pt]
&<\binom{n-2}{d-1}.
\end{align*}

 \vspace{5pt}
 {\bf Case 2. } $(1,2,4,6,8,\ldots,2d)\in \hf^{(d+1)}$.
 \vspace{5pt}

Choose the minimal $j$ such that $(1,2,j,j+2,j+4,\ldots,j+2d-4)\notin  \hf_0^{(d+1)}$. Clearly we have $5\leq j\leq d+3$.

By minimality of $j$, $(1,2,j-1,j+1,j+3,\ldots,j+2d-5)\in \hf_0^{(d+1)}$. Then the 2-intersecting property implies $(1,3,\ldots,j-2,j,j+2,\ldots,j+2(d-j+3))\notin \hf_1^{(d+1)}$. By initiality $w(F)$  hits  the line $y=x+j-3$ for every $F\in \hf_1^{(d+1)}$. Note that $w(F)$ does not pass point $(0,2)$. It follows that $w(F)$ passes either $(1,1)$ or $(2,0)$. Moreover there are 2 ways from $(0,0)$ to $(1,1)$. Thus by Proposition \ref{prop-2.5},
\begin{align}
 \quad |\hf_1^{(d+1)}|&\leq  2\times\mbox{the number of walks from }(1,1) \mbox{ to }(n-d-1,d+1)\nonumber\\[2pt]
 &\qquad \qquad \mbox{ hitting the line }y=x+j-3\nonumber\\[5pt]
 &\qquad +\mbox{the number of walks from }(2,0) \mbox{ to }(n-d-1,d+1)\nonumber\\[2pt]
 &\qquad \qquad \mbox{ hitting the line }y=x+j-3\nonumber\\[5pt]
 &=2\binom{n-2}{d+3-j}+\binom{n-2}{d+2-j}= \binom{n-2}{d+3-j}+\binom{n-1}{d+3-j}.\label{ineq-case2-1}
\end{align}

\begin{claim}
\begin{align}\label{ineq-case2-2}
|\hf_0^{(d+1)}|\leq \binom{n-2}{d-1} - \binom{n-j+2}{d-1}+ \binom{n-j+2}{d-2}.
\end{align}
\end{claim}
\begin{proof}
Since $(1,2,j,j+2,j+4,\ldots,j+2d-4)\notin  \hf_0^{(d+1)}$,
\[
(j-2,j,j+2,\ldots,j+2d-6)\notin  L_2(\hf_0^{(d+1)}([2])).
\]
By  \eqref{ineq-walks2} it follows that
\begin{align*}
|\hf_0^{(d+1)}|=|L_2(\hf^{(d+1)}([2]))| \leq \binom{n-2}{d-1}-\binom{n-j+2}{d-1}+\binom{n-j+2}{d-2}.
\end{align*}
\end{proof}

By Fact \ref{ineq-1.8}, we know that $\hf^{(d+1)},\hf^{(d+i)}$ are cross $(i+1)$-intersecting. Since $(1,2,j-1,j+1,j+3,\ldots,j+2d-5)\in \hf_0^{(d+1)}$,
\[
(1,2,3,\ldots,j+2i-8,j+2i-7,j+2i-6,j+2i-4,\ldots,2d-j+6)\notin \hf^{(d+i)}.
\]
By \eqref{ineq-walks1},
\begin{align*}
 |\hf^{(d+i)}|&\leq \binom{n}{d+i-(j+2i-6)-1}=\binom{n}{d+5-i-j}.
\end{align*}
Since $n\geq 6d$ implies
\[
\frac{\binom{n}{d+5-i-j}}{\binom{n}{d+4-i-j}}= \frac{n-d-4+i+j}{d+5-i-j}\geq 5,
\]
we have
\begin{align}\label{ineq-case2-3}
\sum_{2\leq i\leq d} |\hf^{(d+i)}| &<\binom{n}{d+3-j}+\frac{5}{4}\binom{n}{d+2-j}\nonumber\\[5pt]
&<\binom{n}{d+3-j}+\frac{5}{4}\frac{n}{n-d}
\binom{n-1}{d+2-j}\nonumber\\[5pt]
&< \binom{n}{d+3-j}+2\binom{n-1}{d+2-j}.
\end{align}

Adding \eqref{ineq-case2-1}, \eqref{ineq-case2-2} and \eqref{ineq-case2-3},
\[
|\hf\setminus \hk(n,2d)|= \sum_{1\leq i\leq d}|\hf^{(d+i)}| \leq  3\binom{n}{d+3-j}+\binom{n-2}{d-1} - \binom{n-j+2}{d-1}+ \binom{n-j+2}{d-2}.
\]
We are left to show that
\begin{align}\label{ineq-case2-4}
\binom{n-j+2}{d-1}-\binom{n-j+2}{d-2}-3\binom{n}{d+3-j} > 0.
\end{align}
Since $n\geq 6d$ and $j\leq d+3$ imply
\[
\frac{\binom{n-j+2}{d-1}}{\binom{n-j+2}{d-2}}=\frac{n-j-d+4}{d-1} > 4,
\]
it follows that $\binom{n-j+2}{d-1}-\binom{n-j+2}{d-2}> \frac{3}{4}\binom{n-j+2}{d-1}$. By \eqref{ineq-key}, $5\leq j\leq d+3$ and $n\geq 6d$,
\begin{align*}
\frac{\binom{n-j+2}{d-1}}{\binom{n}{d+3-j}}\geq \left(\frac{n-d}{n-j+3}\right)^{j-2}\left(\frac{n-d-1}{d-1}\right)^{j-4}
\geq 5^{2j-6}/6^{j-2}.
\end{align*}
For $j\geq 6$,
\begin{align*}
\frac{3}{4}\binom{n-j+2}{d-1}- 3\binom{n}{d+3-j}&\geq \frac{3}{4}\binom{n}{d+3-j} \left(5^{2j-6}/6^{j-2}-4\right)\\[5pt]
&\geq \frac{3}{4}\binom{n}{d+3-j} \left(5^6/6^4-4\right)>0.
\end{align*}
Thus \eqref{ineq-case2-4} holds.

For $j=5$, by $n\geq 6d$ we have
\begin{align*}
 &\quad \ \binom{n-3}{d-1}+ \binom{n-3}{d-2}- 3\binom{n}{d-2}\\[5pt]
 &= \binom{n}{d-2} \left(\frac{(n-d+2)(n-d+1)(n-d)(n-d-1)}{(d-1)n(n-1)(n-2)}+\frac{(n-d+2)(n-d+1)(n-d)}{n(n-1)(n-2)}-3\right)\\[5pt]
  &\geq \binom{n}{d-2} \left(\frac{(n-d+1)(n-d)^3}{(d-1)n(n-1)(n-2)}+\frac{(n-d+2)(n-d+1)(n-d)}{n(n-1)(n-2)}-3\right)\\[5pt]
&\geq \binom{n}{d-2} \left(\frac{5^4}{6^3}+\frac{5^3}{6^3}-3\right)>0.
\end{align*}
Thus \eqref{ineq-case2-4} holds and the theorem is proven.
\end{proof}

\section{Proof of Theorem \ref{thm-main2}}
We are going to derive Theorem \ref{thm-main2} from the following proposition.

\begin{prop}\label{lem-3.1}
Let $\hf_i\subset \binom{[n]}{d+1+i}$, $i=1,2,\ldots,d$. Suppose that $\hf_i$ is $(2i+1)$-intersecting and $\hf_i,\hf_j$ are cross $(i+j+1)$-intersecting. Then for $n\geq  9d$,
\[
\sum_{1\leq i\leq d} |\hf_i|\leq \binom{n-3}{d-1}.
\]
\end{prop}

\begin{proof}
Assume that $\hf_i$ is initial for all $i=1,2,\ldots,d$.  Let $\hf_{1,0}=\{F\in \hf_1\colon \{1,2,3\}\subset F\}$ and let $\hf_{1,1}=\hf_1\setminus \hf_{1,0}$. We distinguish three cases.

 \vspace{5pt}
 {\bf Case 1. } $(1,2,3,d+4,d+6,\ldots,3d)\in \hf_{1}$.
 \vspace{5pt}

The 3-intersecting property  implies  $(1,2,4,5,\ldots,d+3)\notin \hf_{1}$. It follows that every $F\in \hf_1$ contains $\{1,2,3\}$. Also, for $i\geq 2$,
 \[
 |\{1,2,3,d+4,d+6,\ldots,3d\}\cap [d+1+i]| < i+2.
 \]
Since $\hf_1,\hf_i$ are cross $(i+2)$-intersecting,  $\hf_i=\emptyset$ for all $i=2,3,\ldots,d$.  Thus,
\[
\sum_{1\leq i\leq d} |\hf_i|\leq |\hf_{1}|\leq  \binom{n-3}{d-1}.
\]

 \vspace{5pt}
 {\bf Case 2. } $(1,2,3,5,7,9,\ldots,2d+1)\notin \hf_{1}$.
 \vspace{5pt}

Then $(2,4,6,\ldots,2d-2)\notin L_3(\hf_1([3]))$.
By \eqref{ineq-walks1} we have
\begin{align}\label{ineq-thm2case2-1}
|\hf_{1,0}|=|L_3(\hf_1([3]))|\leq \binom{n-3}{d-2}.
\end{align}

Since $\hf_1$ is 3-intersecting, by initiality $\hf_1([3]\setminus \{\ell\},[3])$ is 2-intersecting, $\hf_1(\{\ell\},[3])$ is 4-intersecting for  $\ell=1,2,3$ and $\hf_1(\emptyset,[3])$ is 6-intersecting. By \eqref{ineq-universalbound2} we obtain that
\begin{align}
 |\hf_{1,1}|&=  \sum_{1\leq \ell\leq 3}|\hf_1([3]\setminus \{\ell\},[3])|+\sum_{1\leq \ell\leq 3}|\hf_1(\{\ell\},[3])|+|\hf_1(\emptyset,[3])|\nonumber\\[5pt]
 &\leq 3\binom{n-4}{d-2}+3\binom{n-4}{(d+1)-4}+\binom{n-4}{(d+2)-6}\nonumber\\[5pt]
  &= 3\binom{n-4}{d-2}+3\binom{n-4}{d-3}+\binom{n-4}{d-4}\nonumber\\[5pt]
 &=\binom{n-4}{d-2}+\binom{n-3}{d-2}+\binom{n-2}{d-2}.\label{ineq-thm2case2-2}
\end{align}

Since $\hf_i$ is $(2i+1)$-intersecting, by Proposition \ref{prop-2.4}
\[
|\hf_i| \leq \binom{n}{d+1+i-2i-1} =\binom{n}{d-i}.
\]
Note that $n\geq 7d$ implies that
\[
\frac{\binom{n}{d-i}}{\binom{n}{d-i-1}} =\frac{n-d+i+1}{d-i}> 6.
\]
Since $\binom{n}{d-3}=\frac{n}{n-d+3}\binom{n-1}{d-3}<\frac{7}{6}\binom{n-1}{d-3}$,
it follows that
\begin{align}\label{ineq-thm2case2-3}
\sum_{2\leq i\leq d} |\hf_i| \leq \sum_{2\leq i\leq d} \binom{n}{d-i} \leq \binom{n}{d-2}+\frac{6}{5}\binom{n}{d-3}< \binom{n}{d-2}+2\binom{n-1}{d-3}.
\end{align}
Adding  \eqref{ineq-thm2case2-1}, \eqref{ineq-thm2case2-2} and \eqref{ineq-thm2case2-3},
\[
\sum_{1\leq i\leq d}|\hf_i| \leq  \binom{n-4}{d-2} +2\binom{n-3}{d-2}+\binom{n-2}{d-2}+\binom{n}{d-2}+2\binom{n-1}{d-3}\leq  5\binom{n}{d-2}.
\]
By  $n\geq 8d$,
\begin{align*}
\sum_{1\leq i\leq d}|\hf_i| \leq 5\binom{n}{d-2}&= \frac{5 n(n-1)(d-1)}{(n-d+2)(n-d+1)(n-d)} \binom{n-2}{d-1}\\[5pt]
&\leq \frac{5 n^2d}{(n-d)^3}\binom{n-2}{d-1}\\[5pt]
&< \binom{n-2}{d-1}.
\end{align*}

 \vspace{5pt}
 {\bf Case 3. } $(1,2,3,d+4,d+6,\ldots,3d)\notin \hf_{1}$ and $(1,2,3,5,7,9,\ldots,2d+1)\in \hf_{1}$.
 \vspace{5pt}

Choose the minimal $j$ such that $(1,2,3,j,j+2,j+4,\ldots,j+2d-4)\notin  \hf_{1,0}$. Then $6\leq j\leq d+4$.

By minimality of $j$, $(1,2,3,j-1,j+1,j+3,\ldots,j+2d-5)\in \hf_{1,0}$. Then 3-intersecting property implies $(1,2,4,5,\ldots,j-2,j,j+2,\ldots,j+2(d-j+3),j+2(d-j+4))\notin \hf_{1,1}$. By initiality $w(F)$  hits the line $y=x+j-3$ for every $F\in \hf_{1,1}$. Note that $w(F)$ does not pass point $(0,3)$. It follows that $w(F)$ passes exactly one of $(1,2), (2,1), (3,0)$. Moreover there are 3 ways from $(0,0)$ to $(1,2)$ or $(2,1)$. Thus by Proposition \ref{prop-2.5},
\begin{align}
 |\hf_{1,1}|&\leq  3\times \mbox{the number of walks from }(1,2) \mbox{ to }(n-d-2,d+2)\nonumber\\[2pt]
 &\qquad \qquad \mbox{ hitting the line }y=x+j-3\nonumber\\[5pt]
 &\quad + 3\times \mbox{the number of walks from }(2,1) \mbox{ to }(n-d-2,d+2)\nonumber\\[2pt]
 &\qquad \qquad \mbox{ hitting the line }y=x+j-3\nonumber\\[5pt]
 &\quad +\mbox{the number of walks from }(3,0) \mbox{ to }(n-d-2,d+2)\nonumber\\[2pt]
 &\qquad \qquad \mbox{ hitting the line }y=x+j-3\nonumber\\[5pt]
 &=3\binom{n-3}{d+4-j}+3\binom{n-3}{d+3-j}+\binom{n-3}{d+2-j}\nonumber\\[5pt]
 &<3\binom{n-1}{d+4-j}.\label{ineq-thm2case3-1}
\end{align}

Since  $(1,2,3,j,j+2,j+4,\ldots,j+2d-4)\notin \hf_{1}$,
\[
(j-3,j-1,j+1,\ldots,j+2d-7)\notin L_3(\hf_{1}([3])).
\]
By \eqref{ineq-walks2},
\begin{align}\label{ineq-thm2case3-2}
|\hf_{1,0}| =|L_3(\hf_{1}([3]))|\leq \binom{n-3}{d-1} - \binom{n-j+2}{d-1}+ \binom{n-j+2}{d-2}.
\end{align}

Since $\hf_1,\hf_i$ are cross $(i+2)$-intersecting, by $(1,2,3,j-1,j+1,j+3,\ldots,j+2d-5)\in \hf_{1,0}$ we have
\[
(1,2,3,\ldots,j+2i-8,j+2i-7,j+2i-6,j+2i-4,\ldots,2d-j+8)\notin \hf_i.
\]
By \eqref{ineq-walks1} for $2\leq i\leq d$,
\begin{align*}
 |\hf_i|&\leq  \binom{n}{(d+1+i)-(j+2i-6)-1}=\binom{n}{d+6-i-j}.
\end{align*}
Since $n\geq 9d$ implies
\[
\frac{\binom{n}{d+6-i-j}}{\binom{n}{d+5-i-j}}= \frac{n-d-5+i+j}{d+6-i-j}\geq 8,
\]
we infer that
\begin{align}\label{ineq-thm2case3-3}
\sum_{2\leq i\leq d} |\hf_i| \leq  \binom{n}{d+4-j}+\frac{8}{7}\binom{n}{d+3-j}<  \binom{n}{d+4-j}+2\binom{n-1}{d+3-j}.
\end{align}
Adding \eqref{ineq-thm2case3-1},  \eqref{ineq-thm2case3-2} and  \eqref{ineq-thm2case3-3},
\[
 \sum_{1\leq i\leq d} |\hf_i| \leq 4\binom{n}{d+4-j}+\binom{n-3}{d-1} - \binom{n-j+2}{d-1}+ \binom{n-j+2}{d-2}.
\]
We are left to show that
\begin{align}\label{ineq-thm2case3-4}
 \binom{n-j+2}{d-1}- \binom{n-j+2}{d-2}-4\binom{n}{d+4-j} > 0.
\end{align}

Since $n\geq 9d$ implies that
\[
\frac{\binom{n-j+2}{d-1}}{\binom{n-j+2}{d-2}}=\frac{n-j-d+4}{d-1} > 7,
\]
it follows that $\binom{n-j+2}{d-1}-\binom{n-j+2}{d-2}\geq \frac{6}{7}\binom{n-j+2}{d-1}$. As $j\geq 6$ and $n\geq 9d$, by \eqref{ineq-key}
\begin{align*}
\frac{\binom{n-j+2}{d-1}}{\binom{n}{d+4-j}}\geq \left(\frac{n-d-1}{n-j+3}\right)^{j-2} \left(\frac{n-d-2}{d-1}\right)^{j-5}&\geq \left(\frac{n-d-1}{n-3}\right)^{j-2} \left(\frac{n-d}{d}\right)^{j-5}\\[5pt]
&\geq \left(\frac{8}{9}\right)^{j-2} 8^{j-5}= 8^{2j-7}/9^{j-2}.
\end{align*}
Thus,
\begin{align*}
\frac{6}{7}\binom{n-j+2}{d-1}- 4\binom{n}{d+4-j}&\geq \frac{6}{7}\binom{n}{d+4-j} \left(8^{2j-7}/9^{j-2}-\frac{14}{3}\right)\\[5pt]
&\geq \frac{5}{6}\binom{n}{d+4-j} \left(8^{5}/9^{4}-\frac{14}{3}\right)>0
\end{align*}
and the proposition is proven.
\end{proof}

\begin{proof}[Proof of Theorem \ref{thm-main2}]
Let $\hf\subset 2^{[n]}$ be $(2d+1)$-union. Let $\hf_i=\hf^{(d+1+i)}$, $i=1,2,\ldots,d$. By Fact \ref{ineq-1.8}, $\hf_i$ is $(2i+1)$-intersecting and $\hf_i,\hf_j$ are cross $(i+j+1)$-intersecting. By Proposition  \ref{lem-3.1},
\[
\sum_{1\leq i\leq d} |\hf^{(d+1+i)}| = \sum_{1\leq i\leq d} |\hf_i|\leq \binom{n-3}{d-1}.
\]
Since the $(2d+1)$-union property implies that $\hf^{(d+1)}$ is intersecting, by Theorem \ref{thm-diveristy} for $n>36(d+1)$ there exists $x\in [n]$ such that
\[
|\hf^{(d+1)}(\bar{x})|\leq \binom{n-3}{d-1}.
\]
Thus,
\[
|\hf\setminus \hk_x(n,2d+1)|\leq \sum_{1\leq i\leq d} |\hf^{(d+1+i)}| + |\hf^{(d+1)}(\bar{x})| \leq  2\binom{n-3}{d-1}.
\]
\end{proof}

\section{Proof of Theorem \ref{prop-1.9}}

We need Hilton's Lemma \cite{Hilton}, which is a useful reformulation of the Kruskal-Katona Theorem \cite{Kruskal,Katona}.

To state it let us recall the definition of the lexicographic order on $\binom{[n]}{k}$. For two distinct sets $F,G\in \binom{[n]}{k}$ we say that $F$ {\it precedes} $G$ if
\[
\min\{i\colon i\in F\setminus G\}<\min\{i\colon i\in G\setminus F\}.
\]
E.g., $(1,7)$ precedes $(2,3)$. For a positive integer $k$ let $\hl(n,k,m)$ denote the first $m$ members of $\binom{[n]}{k}$.

\begin{lem} [\cite{Hilton}]\label{lem-5.1}
 Let $n,a,b$ be positive integers, $n>a+b$. Suppose that $\ha\subset \binom{[n]}{a}$ and $\hb\subset \binom{[n]}{b}$ are cross-intersecting. Then $\hl(n,a,|\ha|)$ and $\hl(n,b,|\hb|)$ are cross-intersecting as well.
 \end{lem}

 For $\hf\subset \binom{[n]}{k}$ and $0\leq \ell< k$, define the {\it $\ell$th shadow} $\partial^{(\ell)}\hf$ as
\[
\partial^{(\ell)}\hf = \left\{E\in \binom{[n]}{\ell}\colon \mbox{ there exists }F\in \hf \mbox{ such that } E\subset F\right\}.
\]

\begin{prop}[\cite{Sperner}]\label{sperner}
For $\hf\subset \binom{[n]}{k}$ and $0\leq \ell<k$,
\begin{align}\label{ineq-sperner}
\frac{|\partial^{(\ell)}\hf|}{\binom{n}{\ell}} \geq \frac{|\hf|}{\binom{n}{k}}.
\end{align}
\end{prop}

 We  need the following consequence of Sperner's bound (Proposition \ref{sperner}).

 \begin{cor}
 Let $\ha\subset \binom{[n]}{a}$, $\hb\subset \binom{[n]}{b}$ be cross-intersecting families. Then
 \begin{align}\label{ineq-sperner2}
 \frac{|\ha|}{\binom{n}{a}}+ \frac{|\hb|}{\binom{n}{b}}\leq 1.
 \end{align}
 \end{cor}
 \begin{proof}
 Let $\ha^{c}=\{[n]\setminus A\colon A\in \ha\}$. Since $\ha,\hb$ are cross-intersecting, $\partial^{(b)} \ha^{c}\cap \hb=\emptyset$. It follows that
 \[
 |\partial^{(b)} \ha^{c}|+|\hb|\leq \binom{n}{b}.
 \]
 By Sperner's bound (Proposition \ref{sperner}), we have
 \[
 \frac{|\partial^{(b)} \ha^{c}|}{\binom{n}{b}} \geq \frac{|\ha^{c}|}{\binom{n}{n-a}}=  \frac{|\ha|}{\binom{n}{a}}.
 \]
 Thus \eqref{ineq-sperner2} follows.
 \end{proof}

\begin{proof}[Proof of Proposition \ref{prop-1.9}]
Let $\hf\subset 2^{[n]}$ be an initial $2r$-union family. Since adding a subset $F_0\subset F\in \hf$ to $\hf$ will not affect the $2r$-union property, we may assume that $\hf$ is a complex.

\begin{claim}\label{claim-5.1}
If $|\hf^{(r)}|>\binom{n-1}{r-1}+\binom{n-2}{r-1}$, then
\begin{align}\label{ineq-claim5.1}
|\hf^{(r)}|+2.5|\hf^{(r+1)}|\leq \binom{n}{r}.
\end{align}
\end{claim}
\begin{proof}
By Fact \ref{ineq-1.8}, $\hf^{(r)}$ and $\hf^{(r+1)}$ are cross-intersecting. Let $\ha=\hl(n,r, |\hf^{(r)}|)$ and $\hb=\hl(n,r+1, |\hf^{(r+1)}|)$. By Hilton's Lemma (Lemma \ref{lem-5.1}) $\ha,\hb$ are cross-intersecting as well. Since $|\ha|=|\hf^{(r)}|> \binom{n-1}{r-1}+\binom{n-2}{r-1}$, we see that
\[
\left\{A\in \binom{[n]}{r}\colon A\cap \{1,2\}\neq \emptyset\right\}\subset \ha\mbox{ and thereby }|\hb|=|\hb(1,2)|.
\]
Then $\ha(\bar{1},\bar{2})$, $\hb(1,2)$ are cross-intersecting families on the ground set $[3,n]$. By \eqref{ineq-sperner2},
\[
\frac{|\ha(\bar{1},\bar{2})|}{\binom{n-2}{r}}+\frac{|\hb(1,2)|}{\binom{n-2}{r-1}}\leq 1.
\]
Hence,
\[
|\ha(\bar{1},\bar{2})|+2.5|\hb(1,2)| \leq |\ha(\bar{1},\bar{2})|+\frac{n-r-1}{r}|\hb(1,2)| \leq \binom{n-2}{r}.
\]
By $|\ha(\bar{1},\bar{2})|=|\ha|-\binom{n-1}{r-1}-\binom{n-2}{r-1}$ and $|\hb|=|\hb(1,2)|$, \eqref{ineq-claim5.1} follows.
\end{proof}

Let us distinguish two cases.

\vspace{5pt}
{\bf Case 1.} $(2,4,6,\ldots,2r)\notin \hf^{(r)}$.
\vspace{5pt}

By \eqref{ineq-walks1},
\begin{align}\label{ineq-new5-1}
|\hf^{(r)}|\leq\binom{n}{r-1}.
\end{align}
Using that $\hf$ is a complex, for $i=1,2,\ldots,r$ we have
\[
(1,2,\ldots,2i,2i+2,\ldots,2r)\notin \hf^{(r+i)}.
\]
Then by \eqref{ineq-walks1},
\begin{align}\label{ineq-new5-2}
|\hf^{(r+i)}|\leq\binom{n}{r-i-1}.
\end{align}

Since $n>3r$, $\binom{n}{r-i}/\binom{n}{r-i-1} =\frac{n-r+i+1}{r-i}>2$. Thus \eqref{ineq-new5-1} and \eqref{ineq-new5-2}  imply
\begin{align}\label{ineq-5-5}
\sum_{i\geq 0}|\hf^{(r+i)}|< 2\binom{n}{r-1}=2\frac{r}{n-r+1} \binom{n}{r}<\binom{n}{r}.
\end{align}

\vspace{5pt}
{\bf Case 2.} $(2,4,6,\ldots,2r)\in \hf^{(r)}$.
\vspace{5pt}

By Fact \ref{ineq-1.8}, $\hf^{(r+i)}$ is $2i$-intersecting for $i=1,2,\ldots,r$. By \eqref{ineq-ekr} and $n> 3r$ we have
\begin{align}\label{ineq-5-1}
|\hf^{(r+1)}|\leq \binom{n-2}{r-1}.
\end{align}
Using $n> 2r$ and \eqref{ineq-universalbound2},
\begin{align*}
|\hf^{(r+i)}|\leq \binom{n-1}{r-i} \mbox{ for }i=2,3,\ldots,r.
\end{align*}
For $n>3.5r$, $\binom{n-1}{r-i}/\binom{n-1}{r-i-1} =\frac{n-r+i}{r-i}>2.5$, implying
\begin{align}\label{ineq-5-22}
\sum_{i\geq 2}|\hf^{(r+i)}|\leq \sum_{i\geq 2}\binom{n-1}{r-i}< \frac{5}{3}\binom{n-1}{r-2}.
\end{align}

We distinguish three subcases according to the size of $\hf^{(r)}$.

\vspace{5pt}
{\bf Subcase 2.1.} $|\hf^{(r)}|< \binom{n-1}{r-1}+\binom{n-2}{r-1}$.
\vspace{5pt}

By \eqref{ineq-5-1} and \eqref{ineq-5-22},
\begin{align*}
|\hf^{(\geq r)}|&=|\hf^{(r)}|+|\hf^{(r+1)}|+\sum_{i\geq 2} |\hf^{(r+i)}|\\[5pt]
&\leq \binom{n-1}{r-1}+2\binom{n-2}{r-1}+\frac{5}{3}\binom{n-1}{r-2}\\[5pt]
&\leq \binom{n-1}{r-1}+2\binom{n-2}{r-1}+2\binom{n-1}{r-2}\\[5pt]
&=\binom{n-1}{r-1}\left(1+\frac{2(n-r)}{n-1}+\frac{2(r-1)}{n-r+1}\right).
\end{align*}
As $n>3.5r$, using $\frac{n-1}{r-1}>3.5$ and $n-r=(n-1)-(r-1)$, we infer
\begin{align*}
1+\frac{2(n-r)}{n-1}+\frac{2(r-1)}{n-r+1}
&=3+2(r-1)\left(\frac{1}{n-r+1}-\frac{1}{n-1}\right)\\[5pt]
&=3+\frac{2(r-1)(r-2)}{(n-1)(n-r+1)}\\[5pt]
&< 3+\frac{2}{3.5\times 2.5}<3.5.
\end{align*}
Thus $|\hf^{(\geq r)}|<3.5\binom{n-1}{r-1} \leq \binom{n}{r}$.

\vspace{5pt}
{\bf Subcase 2.2.} $\binom{n-1}{r-1}+\binom{n-2}{r-1}<|\hf^{(r)}|\leq 2\binom{n-1}{r-1}$.
\vspace{5pt}

If $\sum_{i\geq 2} |\hf^{(r+i)}|\leq 1.5|\hf^{(r+1)}|$ then by Claim \ref{claim-5.1},
\[
|\hf^{(\geq r)}|=|\hf^{(r)}|+\sum_{i\geq 1} |\hf^{(r+i)}|\leq |\hf^{(r)}|+2.5|\hf^{(r+1)}|\leq \binom{n}{r}
\]
and we are done. Thus we may assume that
\begin{align}\label{ineq-5-3}
|\hf^{(r+1)}| <\frac{2}{3}\sum_{i\geq 2} |\hf^{(r+i)}|.
\end{align}
By \eqref{ineq-5-22},
\[
\sum_{i\geq 1} |\hf^{(r+i)}|\leq \frac{25}{9}\binom{n-1}{r-2}.
\]
Thus,
\begin{align*}
|\hf^{(\geq r)}|\leq 2\binom{n-1}{r-1}+\frac{25}{9}\binom{n-1}{r-2}&=\binom{n-1}{r-1}\left(2+\frac{25}{9}\frac{r-1}{n-r+1}\right)\\[5pt]
&<\binom{n-1}{r-1}\left(2+\frac{25}{9}\times\frac{2}{5}\right)\\[5pt]
&<3.2 \binom{n-1}{r-1}<\binom{n}{r}.
\end{align*}

\vspace{5pt}
{\bf Subcase 2.3.} $|\hf^{(r)}|> 2\binom{n-1}{k-1}$.
\vspace{5pt}

By Claim \ref{claim-5.1} we may assume that
\begin{align}\label{ineq-5-33}
|\hf^{(r+1)}|\leq  \frac{2}{3}\sum_{i\geq 2} |\hf^{(r+i)}|.
\end{align}
 Since $|\hf^{(r)}|> 2\binom{n-1}{r-1}=\binom{n}{r}-\binom{n-3+2}{r}+\binom{n-3+2}{r-1}$, by \eqref{ineq-walks2} we infer that $(3,5,7,\ldots,2r+1)\in \hf^{(r)}$.

Choose the maximal $j$ such that $(j,j+2,\ldots,j+2r-2)\in \hf^{(r)}$. Then $(j+1,j+3,\ldots,j+2r-1)\notin \hf^{(r)}$ and $j\geq 3$. Recall that $\hf$ is an initial complex. If $\hf^{(r+1)}=\emptyset$ then $|\hf^{(\geq r)}|=|\hf^{(r)}|\leq \binom{n}{r}$. Thus we may assume that $\hf^{(r+1)}\neq \emptyset$. By \eqref{ineq-5-33}, $\hf^{(r+2)}\neq \emptyset$. Then the initiality implies  $[r+2]\in \hf^{(r+2)}$. Since $\hf^{(r)}$, $\hf^{(r+2)}$ are cross 2-intersecting, $j\leq r$. Hence $3\leq j\leq r$.

Since $(j+1,j+3,\ldots,j+2r-1)\notin \hf^{(r)}$, by \eqref{ineq-walks2}
\begin{align}\label{ineq-5-4}
|\hf^{(r)}|\leq \binom{n}{r}-\binom{n-j+1}{r}+ \binom{n-j+1}{r-1}.
\end{align}

Recall that $\hf^{(r)},\hf^{(r+i)}$ are cross $i$-intersecting. By $(j,j+2,\ldots,j+2r-2)\in \hf^{(r)}$, we infer that
\[
(1,2,\ldots, j+2i-5, j+2i-4,j+2i-3,j+2i-1,\ldots,2r-j+1)\notin \hf^{(r+i)}.
\]
By \eqref{ineq-walks1} we have
\begin{align*}
|\hf^{(r+i)}| \leq  \binom{n}{(r+i)-(j+2i-3)-1} =\binom{n}{r-i-j+2}.
\end{align*}
Since $n>3.5r$ implies $\binom{n}{r-i-j+3}/\binom{n}{r-i-j+2} =\frac{n-r+i+j-2}{r-i-j+3}>2.5=\frac{5}{2}$,
\[
\sum_{2\leq i\leq r} |\hf^{(r+i)}|\leq \sum_{2\leq i\leq r} \binom{n}{r-i-j+2} \leq \frac{5}{3}\binom{n}{r-j}.
\]
By \eqref{ineq-5-3},
\[
\sum_{1\leq i\leq r} |\hf^{(r+i)}| \leq \left(1+\frac{2}{3}\right)\times \frac{5}{3}\binom{n}{r-j}=\frac{25}{9}\binom{n}{r-j}.
\]
Thus,
\[
|\hf^{(\geq r)}| \leq \binom{n}{r}-\binom{n-j+1}{r}+ \binom{n-j+1}{r-1}+\frac{25}{9}\binom{n}{r-j}.
\]
We are left to show that
\begin{align}\label{ineq-finial}
\binom{n-j+1}{r}-\binom{n-j+1}{r-1}-\frac{25}{9}\binom{n}{r-j}>  0.
\end{align}

Since $n> 3.5r$ and $j\leq r$ imply $\binom{n-j+1}{r}/\binom{n-j+1}{r-1}=\frac{n-r-j+2}{r}\geq 1.5$,
\[
\binom{n-j+1}{r}-\binom{n-j+1}{r-1}\geq \frac{1}{3}\binom{n-j+1}{r}.
\]
As $3\leq j\leq r$ and $n>3.5r$, by \eqref{ineq-key}
\begin{align*}
\frac{\binom{n-j+1}{r}}{\binom{n}{r-j}} \geq \left(\frac{n-r+2}{n-j+2}\right)^{j-1}\left(\frac{n-r+1}{r}\right)^{j}&\geq \left(\frac{n-r+2}{n-1}\right)^{j-1}\left(\frac{n-r}{r}\right)^{j}\geq \frac{2.5^{2j-1}}{3.5^{j-1}}.
\end{align*}
 Since $2.5^7/3.5^3> 14$, for $j\geq 4$ we have
\begin{align*}
\frac{1}{3}\binom{n-j+1}{r}-\frac{25}{9}\binom{n}{r-j}
&\geq \frac{1}{3}\binom{n}{r-j}\left(\frac{2.5^7}{3.5^3}- \frac{25}{9}\right)\\[5pt]
&>\frac{1}{3}\binom{n}{r-j}\left(14- \frac{25}{3}\right)>0
\end{align*}
and \eqref{ineq-finial} follows.

For $j=3$, using $n\geq 3.5r+1$ we have $\binom{n-2}{r}/\binom{n-2}{r-1}=\frac{n-r-1}{r}\geq 2.5$. It follows that
\[
\binom{n-2}{r}-\binom{n-2}{r-1}\geq \frac{3}{5}\binom{n-2}{r}.
\]
Thus,
\begin{align*}
\frac{3}{5}\binom{n-2}{r}-\frac{25}{9}\binom{n}{r-3}
&>\binom{n}{r-3}\left(\frac{3}{5}\times \frac{2.5^5}{3.5^2}- \frac{25}{9}\right)>0.
\end{align*}
Thus \eqref{ineq-finial} holds and the theorem is proven.
\end{proof}

\section{Proof of Theorem \ref{thm-1.10}}

\begin{proof}[Proof of Theorem \ref{thm-1.10}]
Assume that $\hf\subset 2^{[n]}$ is an initial $(2r+1)$-union complex. We distinguish two cases.

\vspace{5pt}
{\bf Case 1.} $(1,3,5,\ldots,2r+1)\notin \hf^{(r+1)}$.
\vspace{5pt}

By \eqref{ineq-walks1},
\[
|\hf^{(r+1)}|\leq \binom{n}{r-1}.
\]
Using the fact that $\hf$ is a complex, we infer $(1,2,3,4,5,\ldots,2i-2,2i-1,2i+1,\ldots,2r+1)\notin \hf^{(r+i)}$ for $i=2,3,\ldots,r+1$. By \eqref{ineq-walks1},
\[
|\hf^{(r+i)}|\leq \binom{n}{(r+i)-(2i-1)-1}=\binom{n}{r-i}.
\]
Since $n\geq 4r$ implies $\binom{n}{r-i}/\binom{n}{r-i-1}=\frac{n-r+i+1}{r-i}>3$,
\[
\sum_{i\geq 1}|\hf^{r+i}| \leq \sum_{i\geq 1}\binom{n}{r-i} \leq \frac{3}{2}\binom{n}{r-1} \leq \frac{3rn}{2(n-r+1)(n-r)}\binom{n-1}{r} <\binom{n-1}{r}.
\]

\vspace{5pt}
{\bf Case 2.} $(1,3,5,\ldots,2r+1)\in \hf^{(r+1)}$.
\vspace{5pt}

Choose maximal $j$ such that $(1,j,j+2,\ldots,j+2r)\in \hf^{(r+1)}$. Then the maximality implies $(1,j+1,j+3,\ldots,j+2r+1)\notin \hf^{(r+1)}$. Clearly $j\geq 3$.  Recall the definition of the left-translate
\[
L_1(\hf)=\left\{\{x_1-1,x_2-1,\ldots,x_k-1\}\colon \{x_1,x_2,\ldots,x_k\}\in \hf\right\}.
\]
It follows that $(j,j+2,\ldots,j+2r)\notin L_1(\hf^{(r+1)}(1))$.
Thus by \eqref{ineq-walks2}
\begin{align}\label{ineq-thm1.10case2-1}
|\hf^{(r+1)}(1)| = |L_1(\hf^{(r+1)}(1))| \leq \binom{n-1}{r}-\binom{n-j+1}{r}+\binom{n-j+1}{r-1}.
\end{align}
If $\hf^{(\geq r+2)}=\emptyset$, then by  \eqref{ineq-ekr} for $n> 2(r+1)$,
\[
|\hf^{(\geq r+1)}|=|\hf^{(r+1)}| \leq \binom{n-1}{r}
\]
with equality holding if and only if $\hf^{(r+1)}=\hs(n,r+1,1)$, the full star up to isomorphism. Thus we may assume that $\hf^{(\geq r+2)}\neq\emptyset$. As $\hf$ is a complex, $[r+2]\in \hf$. By Fact \ref{ineq-1.8} we know that $\hf^{(r+1)},\hf^{(r+2)}$ are cross 2-intersecting, implying that $j\leq r+2$.

By Fact \ref{ineq-1.8}, $\hf^{(r+1)}$ is intersecting. Then $(1,j,j+2,\ldots,j+2r)\in \hf^{(r+1)}$ implies that
\[
(2,3,\ldots,j-1,j+1,\ldots,2r-j+5)\notin \hf^{(r+1)}(\bar{1}).
\]
That is, $(1,2,\ldots,j-2,j,\ldots,2r-j+4)\notin L_1(\hf^{(r+1)}(\bar{1}))$. Using \eqref{ineq-walks1},
\begin{align}\label{ineq-thm1.10case2-2}
|\hf^{(r+1)}(\bar{1})|= |L_1(\hf^{(r+1)}(\bar{1}))|\leq \binom{n-1}{r+1-(j-2)-1} = \binom{n-1}{r-j+2}.
\end{align}

Note also  that $\hf^{(r+1)}, \hf^{(r+i)}$ are cross $i$-intersecting. We see that $(1,j,j+2,\ldots,j+2r)\in \hf^{(r+1)}$ implies
\[
(1,2,\ldots,j+2i-6,j+2i-5,j+2i-3,\ldots,2r-j+5)\notin \hf^{(r+i)}.
\]
Hence, by \eqref{ineq-walks1}
\begin{align*}
|\hf^{(r+i)}|\leq   \binom{n}{r+i-(j+2i-5)-1}= \binom{n}{r-i-j+4}.
\end{align*}
Since $\binom{n}{r-i-j+4}/\binom{n}{r-i-j+3} =\frac{n-r+i+j-3}{r-i-j+4}>3$,
\begin{align}\label{ineq-thm1.10case2-3}
\sum_{i\geq 2}|\hf^{(r+i)}|\leq   \sum_{i\geq 2}\binom{n}{r-i-j+4}\leq \frac{3}{2}\binom{n}{r-j+2}.
\end{align}

Adding \eqref{ineq-thm1.10case2-1}, \eqref{ineq-thm1.10case2-2} and \eqref{ineq-thm1.10case2-3}, we obtain that
\[
|\hf^{(\geq r+1)}| \leq \binom{n-1}{r}-\binom{n-j+1}{r}+\binom{n-j+1}{r-1}+\frac{5}{2}\binom{n}{r-j+2}.
\]
We are left to show that
\begin{align}\label{ineq-finial2}
\binom{n-j+1}{r}-\binom{n-j+1}{r-1}-\frac{5}{2}\binom{n}{r-j+2}>0.
\end{align}
Note that $\binom{n-j+1}{r}/\binom{n-j+1}{r-1}= \frac{n-j-r+2}{r}\geq 2$. It implies that $\binom{n-j+1}{r}-\binom{n-j+1}{r-1}\geq \frac{1}{2}\binom{n-j+1}{r}$.
By \eqref{ineq-key}, $3\leq j\leq r+2$ and $n> 4r$,
\begin{align*}
\frac{\binom{n-j+1}{r}}{\binom{n}{r-j+2}} \geq \left(\frac{n-r}{n-j+2}\right)^{j-1}\left(\frac{n-r-1}{r}\right)^{j-2}\geq 3^{2j-3}/4^{j-1}.
\end{align*}
For $j\geq 5$,
\begin{align*}
\binom{n-j+1}{r}-\binom{n-j+1}{r-1}-\frac{5}{2}\binom{n}{r-j+2} &>\frac{1}{2}\binom{n-j+1}{r}-\frac{5}{2}\binom{n}{r-j+2}\\[5pt]
&\geq \frac{1}{2}\binom{n}{r-j+2} \left(3^{2j-3}/4^{j-1}-5\right)\\[5pt]
&\geq \frac{1}{2}\binom{n}{r-j+2} \left(3^{7}/4^{4}-5\right)>0.
\end{align*}
For $j=4$, using $n> 4r$ we have
\begin{align*}
&\quad\ \binom{n-3}{r}-\binom{n-3}{r-1}-\frac{5}{2}\binom{n}{r-2} \\[5pt] &=\binom{n-3}{r}\left(1-\frac{r}{n-r-2}
-\frac{5}{2}\frac{r(r-1)n(n-1)(n-2)}{(n-r+2)(n-r+1)(n-r)(n-r-1)(n-r-2)}\right)\\[5pt]
&\geq\binom{n-3}{r}\left(1-\frac{r}{n-r-2}
-\frac{5}{2}\frac{r(r-1)n(n-1)(n-2)}{(n-r+1)(n-r)(n-r-1)^3}\right)\\[5pt]
&\geq \binom{n-3}{r}\left(1-\frac{1}{3}
-\frac{5}{2}\times\frac{4^3}{3^5}\right)>0.\end{align*}
For $j=3$,
\begin{align*}
&\quad\ \binom{n-2}{r}-\binom{n-2}{r-1}-\frac{5}{2}\binom{n}{r-1} \\[5pt] &=\binom{n-2}{r}\left(1-\frac{r}{n-r-1}
-\frac{rn(n-1)}{(n-r+1)(n-r)(n-r-1)}\right)\\[5pt]
&\geq \binom{n-2}{r}\left(1-\frac{1}{3}
-\frac{4^2}{3^3}\right)>0.
\end{align*}
Thus \eqref{ineq-finial2} holds and the theorem is proven.
\end{proof}

\section{Concluding remarks}

In the present paper we investigated the maximal overflow in the Katona Theorem. We established best possible bounds under linear constraints. However there is still a gap that remains.

{\noindent \bf Open Problem.} Determine the maximal overflow
\[
\sigma(n,u)=\max\left\{\sigma_u(\hf)\colon \hf\subset 2^{[n]},\ \hf \mbox{ initial  and $u$-union}\right\}, n\geq u+2.
\]

It is not hard to show that in the case $n=u+1$, $u$ is even the maximum is attained for $\hf=2^{[u]}$.

To introduce a closely related problem, let us first recall the notation $A\bigtriangleup B=(A\setminus B)\cup (B\setminus A)$ for {\it symmetric difference}.

Let us mention that $2^{[n]}$ endowed with the distance $d(A,B)=|A\bigtriangleup B|$ is a {\it metric space}. The {\it diameter} of $\ha\subset 2^{[n]}$ is defined as $\delta(\ha)= \max\{|A\bigtriangleup A|\colon A,B\in \ha\}$.

Since $|A\bigtriangleup B| \leq |A\cup B|$, $\delta(\ha) \leq u$ for a $u$-union family $\hf$, the following result strengthens the Katona Theorem.

{\noindent\bf Kleitman's Diameter Theorem (\cite{kleitman}).} Let $n>u>0$. Suppose that $\hf\subset 2^{[n]}$ satisfies $\delta(\hf)\leq u$. Then
\begin{align}\label{ineq-kleitman}
|\hf|\leq |\hk(n,u)|.
\end{align}

Unlike in the Katona Theorem, there are many families attaining equality. In fact for every $A\subset 2^{[n]}$, $\hk_A(n,u)=\{K\bigtriangleup A \colon K\in \hk(n,u)\}$ satisfies $\delta(\hk_A(n,u))=u$. For $n>u+1$ Bezrukov \cite{Be} proved that these are the only optimal families.

For $u=2d$ the family $\hk_A(n,u)=\{K\subset [n]\colon |K\bigtriangleup A|\leq d\}$ and it is called the {\it ball} of {\it radius} $d$ with {\it center} $A$. By analogy for $u=2d+1$ we call $\hk_A(n,u)$ the {\it double ball}.

For his proof Kleitman \cite{kleitman} introduced a very useful operation, the {\it down-shift} $D_i$:
\[
D_i(\hf) =\{D_i(F)\colon F\in \hf\},
\]
where
$$D_{i}(F)=\left\{
                \begin{array}{ll}
                  F\setminus\{i\}, & \hbox{if }i\in F \in \hf \hbox{ and } F\setminus \{i\}\notin \hf; \\[5pt]
                  F, & \hbox{otherwise.}
                \end{array}
              \right.
$$

As $\delta(D_i(\hf))\leq \delta(\hf)$, applying the down-shift $D_i$, $1\leq i\leq n$ repeatedly will eventually produce a complex $\widetilde{\hf}$ with $|\widetilde{\hf}|=|\hf|$ and $\delta(\widetilde{\hf})\leq \delta(\hf)$. For a complex $\widetilde{\hf}$, $\delta(\widetilde{\hf})=\max\{|A\cup B|\colon A,B\in \widetilde{\hf}\}$. Consequently, $|\hf|=|\widetilde{\hf}|\leq |\hk(n,\delta(\hf))|$.

For a fixed pair $(n,u)$, let us define the {\it (diametral) overflow} $\kappa_u(\hf)$ as
\[
\kappa_u(\hf) =\min\{|\hf\setminus \hk_A(n,u)|\colon A\subset [n]\}.
\]

The following two facts are easy to prove.

\begin{fact}\label{fact-6.1}
Let $\hf,\hg\subset 2^{[n]}$, $i\in [n]$. Then
\[
|D_i(\hf)\cap D_i(\hg)|\geq |\hf\cap \hg|.
\]
\end{fact}
\begin{fact}\label{fact-6.2}
$D_i(\hk_A(n,u)) =\hk_{A\setminus \{i\}}(n,u)$.
\end{fact}
\begin{cor}\label{cor-6.3}
Let $\hf\subset 2^{[n]}$ be a complex. Then
\[
\sigma_u(\hf)=\kappa_u(\hf).
\]
\end{cor}

Recall the definition of the families:
\begin{align*}
&\hb(n,2d)=\{B\subset [n]\colon |B\setminus [2]|\leq d-1\}
\end{align*}
and
\begin{align*}
&\hg(n,2d+1)=\{G\subset [n]\colon |G\setminus [3]|\leq d-1\}.
\end{align*}

Then Corollary \ref{cor-6.3} implies
\[
\kappa_{2d}(\hb(n,2d)) =\binom{n-2}{d-1},\ \kappa_{2d+1}(\hg(n,2d+1)) =2\binom{n-3}{d-1}.
\]
Based on these examples let us state:

\begin{conj}
Let $\hf\subset 2^{[n]}$ with $\delta(\hf)\leq u$.  There exists an absolute constant $c$ such that for $n>cu$,
\begin{align*}
\kappa_{u}(\hf) \leq \binom{n-2}{d-1} \mbox{ for } u=2d,\ \kappa_{u}(\hf) \leq 2\binom{n-3}{d-1} \mbox{ for } u=2d+1.
\end{align*}
\end{conj}

\end{document}